\documentclass[11pt]{article}

\usepackage[notcite,notref,final]{showkeys}

\usepackage{latexsym, amsfonts, amsmath,amsthm, amssymb,color, ulem}

\numberwithin{equation}{section}

\newcommand{\be}{\begin{equation}}
\newcommand{\ee}{\end{equation}}
\newcommand{\bea}{\begin{eqnarray}}
\newcommand{\eea}{\end{eqnarray}}
\newcommand{\bean}{\begin{eqnarray*}}
\newcommand{\eean}{\end{eqnarray*}}

\newcommand{\brray}{\begin{array}}
\newcommand{\erray}{\end{array}}
\newcommand{\ben}{\begin{equation}{nonumber}}
\newcommand{\een}{\end{equation}{nonumber}}

\theoremstyle{plain}
\newtheorem{theorem}{Theorem}[section]
\newtheorem{lemma}[theorem]{Lemma}

\newtheorem{proposition}[theorem]{Proposition}
\newtheorem{hypothesis}[theorem]{Hypothesis}
\newtheorem{corollary}[theorem]{Corollary}

\theoremstyle{definition}
\newtheorem{definition}[theorem]{Definition}
\newtheorem{example}[theorem]{Example}

\newtheorem{remark}[theorem]{Remark}

\newcommand{\bdfn}{\begin{dfn}}
\newcommand{\bthm}{\begin{thm}}
\newcommand{\blmma}{\begin{lmma}}
\newcommand{\bppsn}{\begin{ppsn}}
\newcommand{\bcrlre}{\begin{crlre}}
\newcommand{\bxmpl}{\begin{xmpl}}
\newcommand{\brmrk}{\begin{rmrk}}
\newcommand{\edfn}{\end{dfn}}
\newcommand{\ethm}{\end{thm}}
\newcommand{\elmma}{\end{lmma}}
\newcommand{\eppsn}{\end{ppsn}}
\newcommand{\ecrlre}{\end{crlre}}
\newcommand{\exmpl}{\end{xmpl}}
\newcommand{\ermrk}{\end{rmrk}}

\newcommand{\IC}{\mathbb{C}}

\newcommand{\cla}{{\cal A}}

\newcommand{\clq}{{\cal Q}}

\newcommand{\End}{\text{End}}

\def\a*{{\cal A}_{h,*}}
\def\B{{\cal B}(h)}
\def\B1{{\cal B}_1(h)}
\def\b{{\cal B}^{\rm s.a.}(h)}
\def\b1{{\cal B}^{\rm s.a.}_1(h)}

\newcommand{\ot}{\otimes}

\newcommand{\raro}{\rightarrow}

\makeatletter              
\let\c@equation\c@theorem  
\makeatother

\def\a*{{\cal A}_{h,*}}
\def\B{{\cal B}(h)}
\def\B1{{\cal B}_1(h)}
\def\b{{\cal B}^{\rm s.a.}(h)}
\def\b1{{\cal B}^{\rm s.a.}_1(h)}

\author{
Pavel Etingof
  \and
Debashish  Goswami
  \and
Arnab  Mandal
  \and
Chelsea  Walton}
  
  \title{Hopf coactions on commutative algebras generated by 
a quadratically independent comodule}

\AtEndDocument{\bigskip{\footnotesize%
  \textsc{Etingof: Department of Mathematics, Massachusetts Institute of Technology, Cambridge, Massachusetts 02151 USA} \par  
  \textit{E-mail address:} \texttt{etingof@math.mit.edu} \par
  \addvspace{\medskipamount}
   \textsc{Goswami: Stat-Math Unit, Indian Statistical Institute, 203, B. T. Road, Kolkata 700 208, India} \par  
  \textit{E-mail address:} \texttt{goswamid@isical.ac.in} \par
  \addvspace{\medskipamount}
   \textsc{Mandal: Stat-Math Unit, Indian Statistical Institute, 203, B. T. Road, Kolkata 700 208, India} \par  
  \textit{E-mail address:} \texttt{arnabmaths@gmail.com} \par
  \addvspace{\medskipamount}
   \textsc{Walton: Department of Mathematics, Temple University, Philadelphia, Pennsylvania 19122 USA} \par  
  \textit{E-mail address:} \texttt{notlaw@temple.edu} \par 
}}

\date{}

\begin{document}

 \maketitle
 
\begin{abstract}
 Let $\cla$ be a commutative unital algebra over an algebraically closed field $k$ of characteristic~$\neq~2$, whose generators  form a  finite-dimensional subspace $V$, with  no nontrivial homogeneous quadratic relations. 
 Let $\clq$ be a Hopf algebra that coacts on $\cla$ inner-faithfully, while leaving  $V$ invariant. We prove that
 $\clq$ must be commutative when either: 
 (i) the coaction preserves a non-degenerate bilinear form on $V$; or (ii) $\clq$ is co-semisimple, finite-dimensional, and char$(k)$~=~0.
\end{abstract}



\section{Introduction}
We work over an algebraically closed field $k$ of characteristic not equal to 2, unless stated otherwise.
In this note, we study quantum group symmetries, i.e. actions or coactions of Hopf algebras on function algebras of  classical spaces.  
Co-actions of {\it genuine} (i.e. noncommutative) Hopf algebras are of interest here, particularly coactions that are {\it inner-faithful}, 
that is, those that do not factor through the coaction of a proper Hopf subalgebra. 
We say that an algebra $\cla$ admits {\it No Quantum Symmetry} when there does not exist an inner-faithful coaction of a 
genuine Hopf algebra $\clq$ on $\cla$. 

Recently, the problem of establishing No Quantum Symmetry has been addressed in  both the algebraic and analytic frameworks. 
In the algebraic framework, the first and last author established that 
there does not exist an inner-faithful coaction of a noncommutative, finite-dimensional, co-semisimple Hopf algebra on a 
commutative domain, when $k = \bar{k}$ and char($k$)~=~0 \cite[Theorem~1.3]{EW_1}. In joint work with Joardar, the second author established  a similar result  in the analytic and infinite-dimensional 
  setting:   there does not exist a genuine 
compact quantum group that coacts inner-faithfully and isometrically (in the sense of \cite{Gos}) on the Frechet-$\ast$-algebra of smooth functions on a compact, connected, 
   smooth Riemannian manifold \cite[Theorem~12.7]{no_qiso}. 
    On the other hand, one obtains inner-faithful coactions of genuine Hopf algebras $\clq$ on function algebras  $\cla$ of classical spaces $X$ if the hypotheses on the theorems above are dropped;
     see \cite[Remark~4.3]{EW_1}, \cite[Theorem~1.2]{EW_2} and \cite[Theorem~3.7]{huang} for counterexamples. 

         The purpose of this paper is to establish No Quantum Symmetry on some  commutative $k$-algebras  as follows.            
                 
\begin{theorem}  \label{thm:main} Take $\cla$  to be a unital commutative $k$-algebra with a finite-dimensional generating subspace $V$ that is quadratically independent \textnormal{[}Definition~\ref{def:quad}\textnormal{]}. Let $\clq$ be a 
  Hopf algebra that coacts on $\cla$ inner-faithfully.  Suppose that one of the following conditions holds:
\begin{enumerate}
\item[\textnormal{(i)}] the coaction of $\clq$ on $\cla$ preserves a non-degenerate bilinear form  on $V$; or
\item[\textnormal{(ii)}] $k$ is algebraically closed of characteristic 0 and $\clq$ is co-semisimple and finite-dimensional.
\end{enumerate}
Then, $\clq$ is a commutative Hopf algebra. 
\end{theorem}

The result above is an algebraic generalization of  \cite[Theorem~3.2]{Gos quad} and \cite[Lemma~12.1]{no_qiso} in the analytic setting.
With the exception of Remark~\ref{remark_1} and Corollary~\ref{cor}, we work in the algebraic framework so that there is no assumption on $\ast$-structures of the underlying algebras 
or on $\ast$-preservation of the coactions.

 \section{Results}
 
 We begin by introducing {\it quadratic independence}, a condition on the comodule algebra $\cla$ that we will impose throughout this work. 
 
\begin{definition} \label{def:quad}
 Let $\cla$ be a unital commutative algebra over $k$ generated by a finite-dimensional subspace $V \subseteq \cla$ . We say that $V$ is {\it quadratically independent} 
 if the natural linear map 
  from  $S^2 V:=(V \otimes V)/(v \otimes w - w \otimes v)_{v,w \in V}$ to $\cla$, sending $v \ot w$ to $vw = wv\in \cla$, is one-to-one. 
 \end{definition}
 
 Let us give an example of quadratic independence.
 
 \begin{example}   Suppose that a Hopf algebra $\clq$ coacts on the algebra $\cla$ of regular functions on a smooth affine algebraic, analytic, or $C^\infty$-manifold $X$, while preserving the ideal $I_x$ of $\mathcal{O}(X)$ consisting of functions vanishing on a point $x\in X$. Then, the cotangent space $T_x^*X=I_x/I_x^2$ is a quadratically independent $ \clq$-comodule. 
\end{example}
  
Now we verify that condition (i) of Theorem~\ref{thm:main} implies that $\clq$ is commutative; see Proposition~\ref{bilin} below. First, we need a preliminary result.

\begin{lemma} \label{lemma_1} 
  Suppose that $V$ is an inner-faithful finite-dimensional  comodule over a Hopf algebra $\clq$,
  and assume that the decomposition $V\otimes V=S^2V\oplus \wedge^2V $ is preserved by this coaction. Here, $\wedge^2V:=(V \otimes V)/(v \otimes w + w \otimes v)_{v,w \in V}$.
  Then, $\clq$ is commutative. 
\end{lemma}

\begin{proof}
Let the coaction $\alpha$ of $\clq$ on $V$ be given by 
    the Faddeev-Reshetikhin-Takhtajan matrix $T \in \End(V)\otimes \clq$; namely, $\alpha(v)=T(v\otimes 1)$. 
    Consider the natural $\clq$-coaction on $V\otimes V$ defined by the matrix $T^{13}T^{23}\in \End(V\otimes V)\otimes \clq$. Here, $T^{13}=\sum_{i,j}E_{ij}\otimes Id \otimes t_{ij}$, 
and $T^{23}=\sum_{i,j}Id \otimes E_{ij}\otimes t_{ij}$, for the elementary matrices $E_{ij}$.
The hypotheses imply that $T^{13}T^{23}$ lies in $(\End(S^2V)\otimes \clq) \oplus (\End(\wedge^2V)\otimes \clq)$,
that is, it commutes with the permutation that flips the two copies of $V$. Thus, $ T^{13}T^{23}=T^{23}T^{13}$, so
matrix elements of $T$ commute with each other. Since matrix elements of $T$ generate $\clq$ (by the inner-faithfulness
of $V$), we obtain that $\clq$ is commutative. \end{proof}

Let us fix the hypothesis below for the rest of the note, unless stated otherwise.
 
 \begin{hypothesis} \label{hyp}
 Let $\cla$ be a unital commutative $k$-algebra with a finite-dimensional, quadratically independent generating subspace $V$. Take $\clq$ be a Hopf algebra that coacts on $\cla$ inner-faithfully, leaving $V$ invariant, that is, for the coaction $\alpha$ of $\clq$ on $\cla$ we get that $\alpha(V) \subseteq V \otimes \clq$.  
    \end{hypothesis}

\begin{lemma} \label{lemma_2} Recall Hypothesis~\ref{hyp}. Then, the $\clq$-coaction preserves $\wedge^2V\subset V\otimes V$.
\end{lemma}

\begin{proof} We have a $\clq$-comodule algebra morphism $TV \raro \cla$, which in degree 2 maps $ V\otimes V $ onto $S^2V\subset \cla$, due to the commutativity of $\cla$ and quadratic independence of $V$. The kernel is then $ \wedge^2V$, and it is thus preserved by the coaction. \end{proof}

\begin{proposition} \label{bilin} Recall Hypothesis~\ref{hyp} and suppose that $\clq$ preserves a nondegenerate bilinear form $B$ on $V$.
Then, $\clq$ is commutative.
\end{proposition}

\begin{proof} 
The form $B$ defines an invariant nondegenerate form $B^2$ on $V\otimes V$ given by $B^2(a \otimes b, c \otimes d)=B(a,d)B(b,c)$. 
By Lemma \ref{lemma_2}, $\wedge^2V$ is invariant under the coaction of $\clq$ on $V \otimes V$. The orthogonal complement
to $\wedge^2V$ in $V\otimes V$ under the form $B^2$ is $S^2V$. Hence, $S^2V\subset V\otimes V$ is also invariant under the coaction of $\clq$ on $V \otimes V$.
Now the result follows from Lemma~\ref{lemma_1}.
\end{proof}

\begin{remark} \label{remark_1}
Suppose that $k=\IC$ and $\clq$ is a Hopf-$\ast$ algebra such that the coaction $\alpha$ preserves a nondegenerate Hermitian form on $V$, that is, $\langle a_0,b_0\rangle a_1^*b_1=\langle a,b\rangle1_{\clq}$,
where we have written $\alpha(a)=a_0 \otimes a_1$ in Sweedler notation.
Then, we can adapt the proof of Proposition \ref{bilin} to get that $\clq$ is commutative. 
\end{remark}

\begin{corollary} \label{cor}
If $\clq$ is the Hopf-$\ast$ algebra associated to a compact quantum group satisfying  Hypothesis~\ref{hyp}, then $\clq$ must be commutative.
\end{corollary}

\begin{proof}
   By  applying  
  \cite[Proposition~6.4]{van} to the finite-dimensional $\clq$-comodule $V$,
  we get a nondegenerate Hermitian form on $V$ which is preserved by the coaction.  Now $\clq$ is commutative by Remark~\ref{remark_1}.
\end{proof}

Finally, we establish that condition (ii) of Theorem~\ref{thm:main} implies that $\clq$ is commutative.

\begin{theorem}
\label{3}
Recall Hypothesis~\ref{hyp} and assume that  $k$ is algebraically closed of  characteristic $0$. If, further, $\clq$ is finite-dimensional and co-semisimple, then $\clq$  must be commutative.
\end{theorem}
\begin{proof}
Since we have quadratic independence,  we get that 
the action of $\clq^*$ preserves the subspace $\wedge^2V$ in $V\otimes V$ by Lemma~\ref{lemma_2}. 
Hence, $\clq^*$ acts on $TV/(\wedge^2V)=SV$ inner-faithfully. 
Thus,  $\clq^*$ is cocommutative by \cite[Theorem~1.3]{EW_1}. 
\end{proof}

\section*{Acknowledgements}
 P. Etingof and C. Walton were supported by the US National Science Foundation grants  DMS-\#1502244 and  \#1550306. D. Goswami thanks P. Etingof and C. Walton for their kind hospitality at MIT where the work began.

\end{document}